\documentclass[a4paper]{article}

\title{An Algorithm to Solve Polyhedral Convex Set Optimization Problems}

\author{Andreas L\"{o}hne\thanks{Corresponding author. Martin-Luther-Universit\"{a}t Halle-Wittenberg, Institut f\"ur Mathematik, 06099 Halle (Saale), Germany, email:
\texttt{andreas.loehne@mathematik.uni-halle.de}.} \and Carola Schrage}

\date{October 2, 2012 (updated: \today)}

\usepackage{amsmath, amssymb, color, amsthm, bold-extra, mathabx}

\newtheorem{theorem}{Theorem}[section]
\newtheorem{corollary}[theorem]{Corollary}
\newtheorem{remark}[theorem]{Remark}
\newtheorem{lemma}[theorem]{Lemma}
\newtheorem{proposition}[theorem]{Proposition}
\newtheorem{ex}[theorem]{Example}
\newenvironment{example}{\begin{ex} \em }{\em \end{ex}}
\newtheorem{definition}[theorem]{Definition}

\newcommand{\norm}  [1]{\ensuremath{\left  \|       #1  \right \|       }}

\newcommand{\cb}    [1]{\ensuremath{\left  \{      #1  \right \}       }}

\newcommand{\of}    [1]{\ensuremath{\left (        #1  \right )        }}
\newcommand{\ofg}   [1]{\ensuremath{\bigl (        #1  \bigr  )        }}

\newcommand{\st} {\ensuremath{|\;}}

\newcommand{\cl}  {{\rm cl  \,}}
\newcommand{\Int} {{\rm int \,}}
\newcommand{\conv}  {{\rm conv \,}}
\newcommand{\cone}{{\rm cone\,}}
\newcommand{\gr}  {{\rm gr\,}}
\newcommand{\dom}{{\rm dom\,}}
\renewcommand{\subset}{\subseteq}
\renewcommand{\supset}{\supseteq}
\newcommand{\smz}{\setminus\{0\}}

\newcommand{\G}{\mathcal{G}}

\renewcommand{\P}{\mathcal{P}}

\newcommand{\R}{\mathbb{R}}

\begin{document}
\maketitle

\begin{abstract}
An algorithm which computes a solution of a set optimization problem is provided. The graph of the objective map is assumed to be given by finitely many linear inequalities. A solution is understood to be a set of points in the domain satisfying two conditions: the attainment of the infimum and minimality with respect to a set relation. In the first phase of the algorithm, a linear vector optimization problem, called the vectorial relaxation, is solved. The resulting pre-solution yields the attainment of the infimum but, in general, not minimality. In the second phase of the algorithm, minimality is established by solving certain linear programs in combination with vertex enumeration of some values of the objective map.
\\[.2cm]
{\bf Keywords and phrases.} set-valued optimization, set relation, set criterion, vector optimization, infimum attainment
\\[.2cm]
{\bf Mathematical Subject Classification (2000).} 65K05, 90C99
\end{abstract}

\section{Introduction}

In this paper we consider the problem to minimize a set-valued map $F:\R^n\rightrightarrows\R^q$ with polyhedral convex graph with respect to the relation
\[ F(x) \preceq F(u) \quad:\iff\quad F(x) +C \supset F(u), 
\]
where $C$ denotes a polyhedral convex ordering cone that contains no lines and has nonempty interior. The objective map can be considered as a function from $\R^n$ into the space $\G$ of all closed convex subsets of $\R^q$. With the above ordering relation, one obtains a complete lattice, i.e., the infimum $\inf\cb{F(x)\st x\in\R^n}$ in the sense of a greatest lower bound with respect to $\preceq$ always exists. Solution concepts for complete-lattice-valued problems have been introduced in \cite{HeyLoe11}. The main idea is that, beyond scalar optimization, {\em minimality} and {\em infimum attainment} are two different conditions and a solution shall involve both. Such a solution concept is also useful in vector optimization \cite{LoeTam07,HeyLoe11,Loehne11,HamLoeRud12}. 
In a set-valued framework, it has been used, for instance, in \cite{H09,HamLoe12,HamSch12}.
Applications of set optimization based on the above ordering relation and solution concept can be found in mathematical finance in the framework of markets with frictions, see e.g. \cite{HamRud08,HamHey09,HHR11,HamelRudloffYankova12,LoeRud11,HamLoeRud12}. But, in specific calculations, only the infimum attainment have been considered so far. The algorithm presented in this note ensures also minimality.

Optimization problems with set-valued objective function but partially based on other ordering relations or other solution concepts have been investigated by many authors. References and results can be found, for instance, in \cite{Luc88,GoeTamRiaZal03,Jahn04,Hamel05,BotGraWan09,HerRodSam10,Loehne11,JahHa11}.

\section{Preliminaries}

A set $P \subseteq \R^q$ is said to be {\em polyhedral convex} if there is a representation
\begin{equation}\label{eq_H}
P=\bigcap_{i=1}^r \cb{y \in \R^q \st (w^i)^T y \geq \gamma_i} 
\end{equation}
where $w^1, \ldots, w^r \in \R^q\setminus\cb{0}$ and $\gamma_1, \ldots, \gamma_r \in \R$. Equation \eqref{eq_H} is called {\em H-representation} of $P$. Every non-empty polyhedral convex set $P \subseteq \R^q$ can be expressed as a (generalized) convex hull of finitely many points $x^1, \ldots, x^s \in \R^q$ ($s\in \cb{1,2,3,\dots}$) and finitely many directions $d^1, \ldots, d^t \in \R^q\setminus \cb{0}$ ($t \in \cb{0,1,2,\dots}$) through
\begin{equation}\label{eq_V}
P = \cb{ \sum_{i=1}^s \lambda_i x^i + \sum_{j=1}^t \mu_j d^j \bigg|\; \lambda_i \geq 0,\; \sum_{i=1}^s \lambda_i = 1,\; \mu_j \geq 0},
\end{equation}
where $d \in \R^q\smz$ is called a direction of $P$ if $P + \cb{\lambda\cdot d} \subseteq P$ for all $\lambda>0$.
This can be also written with the convex hull of the points and the cone generated by the directions as $P = \conv \cb{x^1, \ldots,x^s} + \cone\cb{d^1, \ldots, d^t}$.
We set $\cone \emptyset = \cb{0}$, thus, $P$ is bounded if and only if $t=0$. Equation \eqref{eq_V} is called a {\em V-representation} of $P$. Numerical methods to compute a V-representation from an H-representation and vise versa are called {\em vertex enumeration}, see e.g. \cite{BarDobHuh96,BreFukMar98}. We denote by $\cl P$ and $\Int P$, respectively, the closure and interior of a set $P \subset \R^q$.

We assume throughout that $C \subset \R^q$ is a pointed (i.e., $C \cap (-C)= \cb{0}$) polyhedral convex cone with $\Int C \neq \emptyset$. The cone $C$ yields a partial ordering $\leq_C$ on $\R^q$ where $y \leq_C v$ is defined by $v-y \in C$. If $C=\R^q_+ :=\cb{y\in \R^q \st y_1 \geq 0, \ldots , y_q \geq 0}$, the component-wise ordering $\leq_{\R^q_+}$ is abbreviated to $\leq$.
The polar cone of $C$ is the set $C^\circ:=\cb{v \in \R^q \st \forall y \in C: v^T y \leq 0}$. A point $y \in \R^q$ is said to be {\em $C$-minimal} in $P\subset \R^q$ if $y \in P$ and $(\cb{y} -C\smz) \cap P = \emptyset$. We assume that an H-representation of $C$ is given, that is, a matrix $Z\in \R^{q \times p}$ such that
\begin{equation}\label{eq_c}
C=\cb{y \in \R^q \st Z^T y \geq 0}.
\end{equation} 
Let $\G$ denote the family of all closed convex subsets of $\R^q$. By $\G_C$ we denote the subfamily of those elements $P$ of $\G$ having the additional property $P=P+C$. For $P,Q \in \G$ we define
\[ P \preceq Q \quad :\iff \quad P + C \supset Q + C,\]
which can be equivalently written as $P + C \supset Q$.
The ordering $\preceq$ is reflexive and transitive in $\G$ (quasi ordering) and, additionally, antisymmetric in $\G_C$ (partial ordering). For $P,Q \in \G$ we define an equivalence relation by
\[ P \sim Q \quad :\iff \quad P + C = Q + C. \]
Clearly, the quotient space $\G/\!\sim$ is isomorphic to $\G_C$ and thus $\preceq$ is a partial ordering in $\G/\!\sim$.

\begin{remark}\label{rem1}
In a theoretical framework the space $\G_C$ is often more convenient and leads to easier formulations. From a computational viewpoint, however, the usage of $\G$ and $\G/\!\sim$ seems to be more natural. This is due to the fact, that an H/V-representation of some $P \in \G$ with $P\subsetneq P+C$ might be known whereas getting an H/V-representation of $P+C$ would require computational effort.
\end{remark}

The partially ordered set $(\G/\!\!\sim,\preceq)$ provides a {\em complete lattice}, i.e., for every subset of $\G/\!\!\sim$ there exist the infimum and supremum, see e.g. \cite{Loehne11} for more details. To simplify the notation we express the infimum and supremum in terms of the (quasi-ordered) space $(\G,\preceq)$, where we have in mind that we actually deal with representatives of equivalence classes. Thus, for nonempty sets $\P \subseteq \G$ we have
\[ \inf \P = \cl\conv \bigcup_{P \in \P}  (P+C) \qquad \sup \P = \bigcap_{P \in \P}  (P+C). \]
Furthermore, we set $\inf \emptyset = \emptyset$ and $\sup \emptyset = \R^q$.
To express minimality we define for $P,Q \in \G$:
\[ P \precneq Q \quad:\iff\quad ( P\preceq Q \text{ and } P\nsim Q).\]

Let $F:\R^n\rightrightarrows\R^q$ be a {\em polyhedral convex} set-valued map, that is, its graph
\[\gr F :=\cb{(x,y)\in \R^n\times\R^q\st y\in F(x)}\]
is a polyhedral convex set. We assume throughout that an H-representation of $\gr F$ is known. This means, there are $A\in\R^{m\times n}$, $B\in \R^{m \times q}$ and $b \in \R^m$ such that
\begin{equation}\label{eq_grH}
\gr F :=\cb{(x,y)\in \R^n\times\R^q \st A x + B y \geq b}.
\end{equation}
The {\em domain} of $F$ is the set $\dom F := \cb{x \in \R^n \st F(x)\neq \emptyset}$.
We consider the following set optimization problem:
\begin{equation}\tag{P}\label{p}
 \text{ minimize } F: \R^n\rightrightarrows\R^q \text{ with respect to } \preceq.
\end{equation}
Problem \eqref{p} is called {\em feasible} if $\dom F \neq \emptyset$. We assume throughout that \eqref{p} is {\em bounded} in the sense that 
\[ \exists v \in \R^q:\; \cb{v} \preceq \inf_{x \in \R^n} F(x). \]
In our setting, it is sufficient for \eqref{p} being bounded that $\dom F$ is a bounded set and 
\[ \forall x \in \R^n,\exists v\in\R^q: \; \cb{v} \preceq F(x),\]
where the latter condition is obviously satisfied for a map $F$ with bounded values $F(x)$. Note that the algorithm introduced below can verify whether the problem is bounded or not.

The following solution concept is based on a combination of minimality and infimum attainment as these notions do no longer coincide in vector and set optimization. It is an adaptation of the concepts introduced in \cite{HeyLoe11,Loehne11,HamLoeRud12} to the present setting.
\begin{definition}
A point $\bar x \in \dom F$ is said to be a {\em minimizer} for \eqref{p} if there is no $x \in \R^n$ with $F(x) \precneq F(\bar x)$. A finite set $\bar X \subseteq \dom F$ is called a {\em finite infimizer} for \eqref{p} if the infimum is attained in $\bar X$, that is,
\[ \inf_{x \in \bar X} F(x) = \inf_{x \in \R^n} F(x). \]
A finite infimizer $\bar X$ of \eqref{p} is called a {\em solution} to \eqref{p} if it consists of only minimizers.
\end{definition}

\section{Vectorial relaxation and pre-solution}

Consider the linear function $f:\R^n\times\R^q \to \R^q$, $f(x,y)=y$. Because of formal reasons we understand $f$ as a set-valued map whose values are singleton sets, i.e.,
\[ f:\R^n\times\R^q\rightrightarrows\R^q, \quad f(x,y)=\cb{y}.\]
The {\em vectorial relaxation} of the set optimization problem \eqref{p} is defined as:
\begin{equation}\tag{VR}\label{r}
 \text{ minimize } f: \R^n\times\R^q\rightrightarrows\R^q \text{ with respect to } \preceq \text{ subject to } y \in F(x).
\end{equation}
Of course, \eqref{r} can be seen as a special case of a set optimization problem, whence the above definitions apply also to \eqref{r}. Obviously, \eqref{r} is feasible if and only if so is \eqref{p}. As $f$ is single-valued and the constraint $y \in F(x)$ can be expressed by finitely many linear inequalities, \eqref{r} is (equivalent to) a linear vector optimization problem. We have
\begin{equation}\label{eq_inf}
 \inf_{x \in \R^n} F(x) = \inf_{x \in \R^n} \inf_{y \in F(x)} \cb{y} = \inf_{x\in \R^n\!\!,\,y \in F(x)} f(x,y),
\end{equation}
i.e., \eqref{p} and \eqref{r} have the same infima. This implies that \eqref{p} is bounded if and only if so is \eqref{r}. Equation \eqref{eq_inf} motivates the following concept.
\begin{definition}
   A finite set $\{x^i \in \R^n \st i=1,\dots,k\}$ is called a pre-solution of \eqref{p} if there exist $y^i \in \R^q$, $i=1,\dots,k$ such that $\{(x^i,y^i)^T \in \R^n\times\R^q \st i=1,\dots,k\}$ is a solution of the vectorial relaxation \eqref{r} of \eqref{p}.
\end{definition}
The following example shows that the $x$-component of a minimizer of \eqref{r} is in general not a minimizer of \eqref{p}. This means that a pre-solution of \eqref{p} is in general not a solution of \eqref{p}.
\begin{example} Consider the set-valued map $F:\R^2 \rightrightarrows \R^2$ where, according to \eqref{eq_grH}, $\gr F \subset \R^4$ is given by
\[
\renewcommand{\arraystretch}{0.8}
A=\left(\begin{array}{rr} 
1 & 0 \\ 
0 & 1 \\ 
-1 & 0\\ 
0 & -1\\ 
0 & 0 \\ 
0 & 0 \\ 
0 & 0 \\ 
0 & 0 \\ 
0 & 1 \\ 
0 & -1\\ 
-2 & 2\\ 
-1 & 1\\ 
2 & -2\\ 
1 & -1   
\end{array}\right)\quad
B=\left(\begin{array}{rr}
 0 & 0 \\ 
 0 & 0 \\ 
 0 & 0 \\ 
 0 & 0 \\ 
-1 & -1\\ 
 2 & 1 \\ 
 1 & 2 \\ 
 1 & 1 \\ 
 1 & 0 \\ 
 0 & 1 \\ 
 1 & 0 \\ 
 1 & 0 \\ 
 0 & 1 \\ 
 0 & 1    
 \end{array}\right)\quad
b=\left(\begin{array}{r}
-1 \\ 
0  \\ 
-1 \\ 
-1 \\ 
-3 \\ 
2  \\ 
2  \\ 
2  \\ 
0  \\ 
-1 \\ 
0  \\ 
0  \\ 
-2 \\ 
-1    
\end{array}\right)\quad
\]
Setting
\[ x^1:=(-1,0)^T,\; x^2:=(-1,1)^T,\; x^3:=(0,0),\; x^4:=(0,1)^T,\]
we see that
\[
\renewcommand{\arraystretch}{1.6}
\begin{array}{l}
 F(x^1)= \conv\cb{\of{0,2}^T,\of{2,0}^T,\of{0,3}^T,\of{3,0}^T}, \\
 F(x^2)= \conv\cb{\of{0,2}^T,\of{-1,4}^T,\of{1,2}^T},\\
 F(x^3) = \conv\cb{\of{0,2}^T,\of{2,0}^T,\of{0,3}^T,\of{4,-1}^T},\\
 F(x^4) = \conv\cb{\of{0,2}^T,\of{2,0}^T,\of{-1,4}^T,\of{3,0}^T}.
 \end{array}
\]
Consider the problems \eqref{p} and \eqref{r} for the ordering cone $C:=\R^2_+$. Set
\[ y^1:=(0,2)^T,\; y^2:=(2,0)^T,\; y^3:=(-1,4),\; y^4:=(4,-1)^T.\]
The infimum for both problems \eqref{p} and \eqref{r} can be expressed as
\[ \inf_{x \in \R^2} F(x) = \conv\cb{y^1,y^2,y^3,y^4}+\R^2_+ =:Q,\]
where each of the points $y^1,\dots,y^4$ is $\leq_C$-minimal in the set $Q$.
It follows that, for instance, the set $\cb{(x^1,y^1)^T,(x^1,y^2)^T,(x^2,y^3)^T,(x^3,y^4)^T}$ is a solution of \eqref{r}. Hence $\cb{x^1,x^2,x^3}$ is a pre-solution of \eqref{p}.
But $\cb{x^1,x^2,x^3}$ is not a solution of \eqref{p}. Indeed, we have $F(x^3) \precneq F(x^1)$ and $F(x^4) \precneq F(x^2)$ which means $x^1$ and $x^2$ are not minimizers for problem \eqref{p}. Note further that $\cb{x^3,x^4}$ is a solution to \eqref{p}.
\end{example}

The following modification of an example provided by Frank Heyde shows that a solution to \eqref{p} is in general not a pre-solution to \eqref{p}.

\begin{example} Let $C=\R^2_+$ and let $F:\R^2 \rightrightarrows \R^2$ with $\dom F = \{(x_1,x_2) \st x_1 \geq 0,\; x_1 + |x_2| \leq 1\}$ be defined by 
$$ F(x) := \cb{(z_1,z_2) \st z_1\geq -x_1+x_2,\, z_2\geq -x_1-x_2,\, z_1+z_2\geq x_1}.$$
The set $\{(0,1)^T, (0,-1)^T, (\frac{1}{2},0)^T\}$ is a solution, but not a pre-solution. Indeed, $(\frac{1}{2},0)^T$ is a minimizer for \eqref{p}, but no point of the form $(\frac{1}{2},0,y_1,y_2) \in \gr F$ is a minimizer for \eqref{r}. 
\end{example}

As a consequence of the following statement we obtain that every solution to \eqref{p} contains a pre-solution to \eqref{p}.

\begin{proposition}\label{prop:Corr2} Every finite infimizer for \eqref{p} contains a pre-solution to \eqref{p}.
\end{proposition}
\begin{proof}
There are $y^1,\dots, y^r \in \R^q$ such that $\bar P = \inf_{x \in \R^n} F(x)$ holds for the polyhedron $\bar P :=\conv\cb{y^1,...,y^r}+C$.
Without loss of generality we assume that $y^1,\dots, y^r$ are the vertices of $\bar P$. Hence, any set $\cb{(x^i,y^i)\st i=1,...,r}$ with $y^i\in F(x^i)$ is a solution to \eqref{r} and consequently $\cb{x^1,...,x^r}$ is a pre-solution to \eqref{p}. Let $\cb{\bar x^1,...,\bar x^k}$ be a finite infimizer to \eqref{p} and assume that $y^m\notin \bigcup_{j=1,...,k}F(\bar x^j)$ for some $m\in\cb{1,...,r}$. Then $y^m$ is not a vertex of $\inf_{j=1,\dots,k} F(\bar x^j) = \bar P$, a contradiction.
\end{proof}

\section{Algorithm}

According to \eqref{eq_grH} and \eqref{eq_c}, let an instance of problem \eqref{p} be given by $A\in \R^{m\times n}$, $B \in \R^{m\times q}$, $b \in \R^m$, $Z \in \R^{q \times p}$. The algorithm computes a solution to \eqref{p} if the problem is feasible and bounded. Otherwise it detects whether \eqref{p} is infeasible or unbounded. 

In the first phase of the algorithm, the vectorial relaxation \eqref{r}, which is (equivalent to) a linear vector optimization problem, is solved. A solution can be obtained, for instance, with Benson's algorithm, see e.g. \cite{Benson98a,EhrLoeSha12,ShaEhr08,ShaEhr08-1,Loehne11,HamLoeRud12}. We know that \eqref{p} is bounded if and only if so is \eqref{r}. But Benson's algorithm is able to detect if \eqref{r} is unbounded. Note further that in \cite{Loehne11}, Benson's algorithm was extended for unbounded linear vector optimization problems. 

In the second phase, for every point $x^0$ of the pre-solution obtained in the first phase, we  construct a sequence $(x^0,x^1,x^2,\dots,x^l)$ with $F(x^0) \succneq F(x^1) \succneq F(x^2) \succneq \dots\succneq F(x^l)$ until, after finitely many steps, a minimizer $x^l$ is obtained. 

For parameters $w \in \R^q$ and $\bar x \in \dom F$, we consider the following scalar problem:
\begin{equation}\tag{P($w$,$\bar x$)}\label{px}
 \text{ maximize } w^T y \;\text{ subject to }\; y \in F(x), \; F(x) \preceq F(\bar x)  .
\end{equation}
As the $y_1,\dots,y_q\in\R$ are considered to be auxiliary variables, we use the following convention: $\hat x$ is said to be a solution to \eqref{px} if there exists $\hat y\in \R^q$ such that $(\hat x,\hat y)$ is a solution to \eqref{px} in the ordinary sense.
In practice this means that $(\hat x,\hat y)$ is computed but only $\hat x$ is used.
Obviously, we have the following lower bound $\beta$ for the optimal value $\alpha$ of \eqref{px}:
\begin{equation}\label{eq_albe}
\renewcommand{\arraystretch}{1.3}
\begin{array}{ll}
 \alpha(w,\bar x) &:=\sup\cb{w^T y \st y \in F(x),\,F(x) \preceq F(\bar x)} \\
                  &\,\geq \sup\cb{w^T y \st y \in F(\bar x)}=:\beta(w,\bar x).
\end{array}                  
\end{equation}
This leads to an optimality condition.
\begin{lemma}\label{lem2}
Let \eqref{p} be feasible and bounded. For some $\bar x \in \dom F$ let an H-representation of the set $F(\bar x)+C$ be given, that is,
\[ F(\bar x)+C = \cb{y \in \R^q \big|\; \ofg{w^j}^T y \leq \gamma_j,\; j=1,\dots,r}.\]
If $\alpha\ofg{w^j,\bar x}=\beta\ofg{w^j,\bar x}$ for all $j\in \cb{1,\dots,r}$, then $\bar x$ is a minimizer for \eqref{p}.
\end{lemma}
\begin{proof}
Assume that $\bar x$ is not a minimizer for \eqref{p}, i.e., there exists $x \in \R^n$ with $F(x) \precneq F(\bar x)$. Hence there is some $y \in F(x)$ and some $c \in C$ such that $y+c \not\in F(\bar x)+C$. Since $C+C=C$, we conclude that $y \not\in F(\bar x)+C$. Thus there is some $j \in \cb{1,\dots,r}$ such that ${w^j}^T y > \gamma_j$ which implies $\alpha\ofg{w^j,\bar x}> \gamma_j \geq \beta\ofg{w^j,\bar x}$.
\end{proof}

With the aid of a V-representation of the set $F(\bar x)+C$, that is,
\begin{equation}\label{eq_Vrep}
 F(\bar x)+C = \conv\cb{y^1,\dots,y^s}+C,
\end{equation} 
\eqref{px} can be transformed into a linear program. Note that the assumption that \eqref{p} is bounded was used in \eqref{eq_Vrep}, otherwise the cone on the right hand side can be a superset of $C$. By \eqref{eq_grH}, the first constraint $y\in F(x)$ can be expressed as $Ax+By \geq b$. The second constraint can be transformed as follows:
\begin{equation}\label{eq_constr}
\renewcommand{\arraystretch}{1.1}
\begin{array}{lcl} 
   F(x) \preceq F(\bar x) &\iff& F(x)+C \supseteq F(\bar x) + C \\
                               &\iff& \forall i=1,\dots,s:\; y^i \in F(x)+C\\
                               &\iff& \forall i=1,\dots,s,\; \exists c^i \in C: y^i - c^i\in F(x)\\
                               &\iff& \forall i=1,\dots,s,\; \exists c^i \in C: A x - B c^i \geq b - B y^i
\end{array}
\end{equation}
Thus, \eqref{px} is equivalent to the linear program
\begin{equation}\tag{P($w;y^1,\dots,y^s$)}\label{py}
\renewcommand{\arraystretch}{1.1}
 \text{ max } w^T y \;\text{ s.t.}
 \cb{\begin{array}{rll}     
     Ax+By \!\!&\geq b & \\
     Ax - Bc^i \!\!&\geq b - B y^i &(i=1,\dots,s)\\
               Z^T c^i \!\!&\geq 0 &(i=1,\dots,s) 
     \end{array}}
\end{equation}
which has $n+q(s+1)$ variables and $m+ms+ps$ constraints. According to the above convention we will speak about a solution $x\in \R^n$ and do not mention the $q(s+1)$ auxiliary variables $y,c^1,\dots,c^s\in \R^q$. 

\smallskip 
\subsubsection*{Algorithm \bfseries{\scshape{SetOpt}}.}
 
Input:\\
\indent H-representation of $\gr F$ according to \eqref{eq_grH}: $A\in \R^{m\times n}$, $B \in \R^{m \times q}$, $b \in \R^m$;\\
\indent H-representation of the ordering cone according to \eqref{eq_c}: $Z\in \R^{q\times p}$; \\

\noindent Output:\\
\indent A solution $\bar X$ of \eqref{p} if \eqref{p} is feasible and bounded, $\bar X = \emptyset$ otherwise;\\
\indent The solution status for \eqref{p};\\

\noindent Phase 1:\\
\indent $\bar X \leftarrow \emptyset$;\\
\indent solve \eqref{r};\\
\indent {\bf if} \eqref{r} is infeasible {\bf then} $status \leftarrow \text{\em''\eqref{p} is infeasible.''}$; stop; {\bf end}; \\
\indent {\bf if} \eqref{r} is unbounded {\bf then} $status \leftarrow \text{\em''\eqref{p} is unbounded.''}$; stop; {\bf end};\\
\indent store a pre-solution $\cb{x^1,\dots,x^k}$ of \eqref{p};\\
\noindent Phase 2:\\
\indent {\bf for} $i \leftarrow 1$ {\bf to} $k$ {\bf do}\\
\indent\indent flag $\leftarrow$ 1;\\
\indent\indent$K \leftarrow \emptyset$;\\
\indent\indent {\bf while} $flag=1$ {\bf do}\\
\indent\indent\indent compute a V-representation and an H-representation of $F(x^i)+C$:\\
\indent\indent\indent\indent $\begin{array}{ll}
 F(x^i)+C &=\conv\cb{y^1,\dots,y^s}+C\\
          &=\cb{y \in \R^q \big|\; \of{w^j}^T y \leq \gamma_j,\; j=1,\dots,r};
          \end{array}
          $
 \\
\indent\indent\indent {\bf for} $j \leftarrow 1$ {\bf to} $r$ {\bf do}\\
\indent\indent\indent\indent {\bf if} $w^j/\|w^j\| \notin K$ {\bf then}\\
\indent\indent\indent\indent\indent $K \leftarrow K \cup \cb{w^j/\|w^j\|}$;\\
\indent\indent\indent\indent\indent $\text{solve (P($w^j$; $y^1,\dots y^s$))}$; $x_i \leftarrow$ solution; $\alpha \leftarrow$ optimal value;\\ 
\indent\indent\indent\indent\indent $\beta \leftarrow \max\cb{\ofg{w^j}^T y^1,\dots,\ofg{w^j}^T y^s}$;\\
\indent\indent\indent\indent\indent {\bf if} $\alpha > \beta$ {\bf then} break (i.e., exit the inner-most loop);\\
\indent\indent\indent\indent {\bf end}; \\
\indent\indent\indent\indent {\bf if} $j=r$ {\bf then} $flag \leftarrow$ 0; \\
\indent\indent\indent {\bf end}; \\
\indent\indent {\bf end};\\
\indent\indent $\bar X \leftarrow \bar X \cup \cb{x^i}$;\\
\indent {\bf end};\\
\indent $status \leftarrow \text{\em''\eqref{p} has been solved.''}$;

\bigskip
\noindent
We next show that the algorithm works correctly and is finite. We prepare the theorem by two lemmas.

\begin{lemma}\label{lem0}
Let \eqref{p} be feasible and bounded. Consider some $\bar x \in \dom F$, a halfspace $H=\cb{y \in \R^q \st w^T y \leq \gamma}$ containing the set $F(\bar x) + C$ and finitely many points $y^1,\dots,y^s\in \R^q$ such that \eqref{eq_Vrep} holds.
Then, 
\begin{itemize}
\item[(i)] The linear program {\rm(P($w$; $y^1,\dots y^s$))} has an optimal solution; 
\item[(ii)] The lower bound $\beta$ defined in \eqref{eq_albe} can be expressed as
  $$\beta(w,\bar x) = \max\cb{w^T y^1,\dots,w^T y^s}.$$ 
\end{itemize}
\end{lemma}
\begin{proof}
Let $\bar y \in F(\bar x)$. Then, $A \bar x + B \bar y \geq b$. By \eqref{eq_Vrep}, for all $i\in \cb{1,\dots,s}$, we have $y^i \in F(\bar x) +C$, i.e., there is some $c^i\in C$ such that $y^i-c^i\in F(\bar x)$, or equivalently, $A \bar x - B c^i \geq b - B y^i$.
Hence, the point $(\bar x,\bar y,c^1,\dots,c^s)$ is feasible for \eqref{py}.

Since \eqref{p} is assumed to be bounded, there exists $v \in \R^q$ such that $\cb{v} + C \supset F(x)$ for all $x \in \R^n$. As $H$ contains $F(\bar x) + C$, we must have $w \in C^\circ$. It follows that
\[ \sup\cb{w^T y \st x\in\R^n\!, y \in F(x)} \leq w^T v +\sup\cb{w^T c\st c \in C} = w^T v,\]
which implies that \eqref{px} and hence \eqref{py} is bounded. This proves (i).

Statement (ii) follows from \eqref{eq_Vrep} taking into account that $\sup_{c \in C} w^T c = 0$ for $w \in C^\circ$.
\end{proof}

\begin{lemma}\label{lem1}
Let \eqref{p} be feasible and bounded and let $\bar x \in \dom F$. If $\hat x \in \R^n$ is a solution of \eqref{px} for some $w \in C^\circ$, then $\alpha(w,u) = \beta(w,u)$ for every $u\in \R^n$ with $F(u)\preceq F(\hat x)$.
\end{lemma}
\begin{proof}
Obviously, we have $\beta(w,u)\leq \alpha(w,u)$.
For $u\in \R^n$ with $F(u)\preceq F(\hat x)$, we get
\[
\renewcommand{\arraystretch}{1.2}
\begin{array}{ll}
\alpha(w,u) &=\sup\cb{w^T y \st y \in F(x),\,F(x) \preceq F(u)} \\
            &\leq \sup\cb{w^T y \st y \in F(x),\,F(x) \preceq F(\hat x)} = \alpha(w,\hat x).   
\end{array}                 
\]
Since $F(u)\preceq F(\hat x)$ can be written as $F(u) \supseteq F(\hat x)+C$ and, as $w \in C^\circ$, we obtain
\[ \beta(w,u)= \sup\cb{w^T y \st y \in F(u)} \geq \sup\cb{w^T y \st y \in F(\hat x)} = \beta(w,\hat x).\] 
The point $\hat x$ being a solution of \eqref{px} implies that there exists $\hat y\in F(\hat x)$ such that $\beta(w,\hat x)\geq w^T \hat y = \alpha(w,\bar x)\geq\alpha(w,\hat x)$. Altogether we get $\alpha(w,u)\leq \alpha(w,\hat x) \leq \beta(w,\hat x) \leq \beta(w,u)\leq \alpha(w,u)$, which yields the desired equality. 
\end{proof}

\begin{theorem} {\scshape SetOpt} computes a solution of (P) whenever \eqref{p} is feasible and bounded. Otherwise {\scshape SetOpt} states whether \eqref{p} is infeasible or unbounded. 

If the H-representation computed in phase 2 contains no redundant inequalities, {\scshape SetOpt} terminates after finitely many steps. To be more precise, let the pre-solution computed in the first phase consist of $k$ points and let $l$ be the number of linear inequalities necessary to describe the polyhedral convex set $\gr F + (0_{\R^n}\times C)$. In the second phase of {\scshape SetOpt}, at most $l\cdot k$ linear programs have to be solved.
\end{theorem}
\begin{proof}
If \eqref{p} is infeasible or unbounded, so is \eqref{r} and the algorithm terminates with the corresponding status. If \eqref{p} is feasible and bounded, a pre-solution $\cb{x^1,\dots,x^k}$, $k\geq 1$ of \eqref{p} is obtained by solving \eqref{r}.

For fixed $i \in \cb{1,\dots,k}$, denote by $\bar x^i$ the value of the variable $x^i$  before the while loop has been entered and let $\hat x^i$ be the value after the while loop has been left. Then, $\bar X:=\cb{\bar x^1,\dots,\bar x^k}$ is the pre-solution computed in phase 1 and $\hat X:=\cb{\hat x^1,\dots,\hat x^k}$ is the result of the algorithm. We will show that for all $i \in \cb{1,\dots,k}$,
\begin{equation}\label{eq_th41} 
F\ofg{\hat x^i} \preceq F\ofg{\bar x^i} \qquad \text{and} \qquad \not\exists\, x \in \R^n: F(x) \precneq F\ofg{\hat x^i}.
\end{equation}
The first condition in \eqref{eq_th41} implies 
\[ \inf_{x \in \hat X} F(x) \preceq \inf_{x \in \bar X} F(x).\]
Since $\bar X$ is a pre-solution of \eqref{p}, the infimum is attained in $\bar X$, that is,
\[ \inf_{x \in \bar X} F(x) = \inf_{x \in \R^n} F(x).\]
It follows that the infimum is also attained in $\hat X$.
The second condition in \eqref{eq_th41} states that $\hat X$ consists of only minimizers, whence $\hat X$ is a solution to \eqref{p}.
 
To show \eqref{eq_th41}, let $i \in \cb{1,\dots,k}$ be fixed. The first condition of \eqref{eq_th41} follows directly from the constraint $F(x) \preceq F\ofg{x^i}$ of the equivalent formulation (P($w^j,x^i$)) of the linear program (P($w^j$; $y^1,\dots y^s$)). Note that by Lemma \ref{lem0} an optimal solution of (P($w^j$; $y^1,\dots y^s$)) always exists in the case where \eqref{p} is feasible and bounded. The while loop is left only after $flag$ has been set to zero. This requires $r$ iterations in the inner for loop, which occurs only if for every $j \in \cb{1\dots,r}$, $\alpha=\alpha(w^j,\hat x^i)$ equals $\beta=\beta(w^j,\hat x^i)$. In case of $w^j/\|w^j\|\in K$, this is known by Lemma \ref{lem1}, i.e., (P($w^j$; $y^1,\dots y^s$)) does not need to be solved. But $\alpha(w^j,\hat x^i)=\beta(w^j,\hat x^i)$ for all $j \in \cb{1\dots,r}$ implies that $\hat x^i$ is a minimizer, compare Lemma \ref{lem2}. 

To show finiteness, note first that there is a finite algorithm (such as Benson's algorithm) to solve \eqref{r} in phase 1. Consider the map $\tilde F:\R^n\rightrightarrows\R^q$, $\tilde F(x):=F(x)+C$. Of course, $\gr \tilde F = \gr F + (0_{\R^n}\times C)$ is a polyhedral convex set. Thus it can be expressed as
\[ \gr \tilde F = \cb{(x,y)\in \R^n\times\R^q\st \tilde A x + \tilde B y \geq \tilde b},\]
for some $\tilde A \in \R^{l \times n}$, $\tilde B \in \R^{l \times q}$, $\tilde b \in \R^l$.
For every $x \in \dom F$, we have 
\[ F(x)+C=\cb{y \in \R^q \st \tilde B y \geq \tilde b - \tilde A x}.\]
Consequently, every H-representation of $F(x)+C$ that contains no redundant inequalities consists of at most $l$ inequalities. In other words, there are at most $l$ different vectors $w^j/\norm{w^j}$ for $j\in \cb{1,\dots,k}$ in the algorithm, which proves the claim.
\end{proof}

The theorem immediately implies the following existence result.

\begin{corollary} If \eqref{p} is feasible and bounded, a solution exists.
\end{corollary}

To reduce the computational effort of the algorithm for specific problems, we suggest an additional rule. Let $\cb{(x_1,y_1),\dots,(x_k,y_k)}$ denote the solution of \eqref{r} obtained in phase 1 and consider iteration $i\in \cb{1,\dots,k}$ of the outer for loop in phase 2:
\[ \text{\em If $y_j \in F(x_i)$ for $j$ with $i < j \leq k$ then skip all commands in iteration $j$.}\] 
Clearly, this rule maintains the attainment of the infimum and thus the algorithm still works correctly. 

Finally we consider the special situation where $F:\R^n \to \G_C$, i.e., we have $F(x)=F(x)+C$ for all $x \in \R^n$. In this case, \eqref{eq_constr} can be replaced by 
\[
\begin{array}{lcl} 
   F(x) \preceq F(\bar x) &\iff& F(x) \supseteq F(\bar x) \\
                               &\iff& \forall i=1,\dots,s:\; y^i \in F(x)\\
                               &\iff& \forall i=1,\dots,s: A x \geq b - B y^i\\
                               &\iff& A x \geq \max\cb{b - B y^i \st i=1,\dots,s}.
\end{array}
\]
This means that the linear program \eqref{py} has only $n+q$ variables and $2m$ constraints. The problem to obtain an H-representation of $\gr (F(\cdot)+C)$ from an H-representation of $\gr F$ seems to be difficult in practice where it is typical that $n \gg q$, compare also Remark \ref{rem1}. One way to obtain it is vertex enumeration of a polyhedral convex set in $\R^n\times\R^q$. In contrast, {\scshape SetOpt} involves vertex enumeration only in $\R^q$.

We finally show a special property of solutions obtained by {\sc SetOpt}.

\begin{proposition}
  Every solution to \eqref{p} obtained by the algorithm {\sc SetOpt} is a pre-solution to \eqref{p}.
\end{proposition}
\begin{proof} In phase 1 of the algorithm, a pre-solution $\bar X = \cb{\bar x^1,\dots,\bar x^k}$ is computed. In phase 2 of the algorithm, $\bar X$ is replaced by $\hat X = \cb{\hat x^1,\dots,\hat x^k}$, where
	$f(\hat x^i) \preceq f(\bar x^i)$, $i\in \cb{1,\dots,k}$ by \eqref{eq_th41}. Hence, $(\hat x^i, y)$ is a minimizer of \eqref{r} whenever $(\bar x^i, y)$ is a minimizer of \eqref{r}, which proves the claim.
\end{proof}



\end{document}